\newcommand{\Qq}{\enm{\cal{Q}}}

\newcommand{\Tt}{\enm{\cal{T}}}

\renewcommand{\phi}{\varphi}
\renewcommand{\theta}{\vartheta}
\documentclass[12pt,oneside,reqno]{amsart}
 
\usepackage[all]{xy}                        %
\usepackage[hyperindex=true, colorlinks=true, linkcolor=blue, filecolor=black,
citecolor=blue, urlcolor=blue, bookmarks=true, pagebackref=true,
pdftitle={Elementary transformations of toric vector bundles}]{hyperref}

\usepackage[margin=1in]{geometry}
\usepackage[table]{xcolor}
\usepackage{colortbl, array, booktabs}

\definecolor{lightred}{rgb}{1, 0.85, 0.85}
\definecolor{lightblue}{rgb}{0.85, 0.9, 1}
\definecolor{darkline}{rgb}{0.4, 0.4, 0.4} 
\usepackage{verbatim}
\usepackage[numbers]{natbib}
\CompileMatrices                            

\UseTips                                    

\usepackage[bookmarks=true]{hyperref}       
\hypersetup{
    colorlinks=true,    
    linkcolor=blue,     
    citecolor=blue,     
    filecolor=blue,     
    urlcolor=blue       
}
\usepackage{amssymb,latexsym,amsmath,amscd}
\numberwithin{equation}{section}
\usepackage{xspace}
\usepackage{color}
\usepackage{graphicx}
\usepackage{caption}
\usepackage{mathtools}
\usepackage{tikz}
\usetikzlibrary{cd}
\usepackage[utf8]{inputenc}
\usepackage{fourier}
\usepackage{array}
\usepackage{makecell}
\usepackage{colortbl}
\usepackage{bm}
\usepackage{array}
\usepackage{subcaption}
\usepackage[table]{xcolor}
\usepackage{multirow}

\definecolor{lightred}{RGB}{255,220,220}
\definecolor{lightblue}{RGB}{220,235,255}

\newcolumntype{L}[1]{>{\raggedright\let\newline\\\arraybackslash\hspace{0pt}}m{#1}}
\newcolumntype{C}[1]{>{\centering\let\newline\\\arraybackslash\hspace{0pt}}m{#1}}
\newcolumntype{R}[1]{>{\raggedleft\let\newline\\\arraybackslash\hspace{0pt}}m{#1}}
\definecolor{lightred}{rgb}{1.0,0.9,0.9}
\definecolor{lightblue}{rgb}{0.9,0.9,1.0}

\reversemarginpar

\theoremstyle{plain}
\newtheorem{theorem}{Theorem}[section]
\newtheorem*{theorem*}{Theorem}
\newtheorem{proposition}[theorem]{Proposition}
\newtheorem{corollary}[theorem]{Corollary}

\theoremstyle{definition}
\newtheorem{definition}[theorem]{Definition}

\newtheorem{remark}[theorem]{Remark}
\newtheorem{example}[theorem]{Example}

\newcommand{\enm}[1]{\ensuremath{#1}}          %
\newcommand{\op}[1]{\operatorname{#1}}
\newcommand{\cal}[1]{\mathcal{#1}}

\renewcommand{\bar}[1]{\overline{#1}}

\newcommand{\CC}{\enm{\mathbb{C}}}

\newcommand{\RR}{\enm{\mathbb{R}}}

\newcommand{\ZZ}{\enm{\mathbb{Z}}}
\newcommand{\FF}{\enm{\mathbb{F}}}

\newcommand{\PP}{\enm{\mathbb{P}}}

\newcommand{\Ee}{\enm{\cal{E}}}
\newcommand{\Ff}{\enm{\cal{F}}}
\newcommand{\Gg}{\enm{\cal{G}}}
\newcommand{\Hh}{\enm{\cal{H}}}
\newcommand{\Ii}{\enm{\cal{I}}}

\newcommand{\Kk}{\enm{\cal{K}}}
\newcommand{\Ll}{\enm{\cal{L}}}

\newcommand{\Nn}{\enm{\cal{N}}}
\newcommand{\Oo}{\enm{\cal{O}}}
\renewcommand{\epsilon}{\varepsilon}

\newcommand{\Hom}{\op{Hom}}

\newcommand{\gr}{\op{gr}}

\renewcommand{\to}[1][]{\xrightarrow{\ #1\ }}

\newcommand{\old}[1]{}

\makeatletter
\NewCommandCopy\@@pmod\pmod
\DeclareRobustCommand{\pmod}{\@ifstar\@pmods\@@pmod}
\def\@pmods#1{\mkern4mu({\operator@font mod}\mkern 6mu#1)}
\makeatother

\begin{document}
\renewcommand{\arraystretch}{1.4}
\arrayrulecolor{darkline}

\title[Elementary transformations of toric vector bundles]{Elementary transformations of toric vector bundles}

\author{Sukmoon Huh}


\address{Sungkyunkwan University, Suwon 440-746, Korea}
\email{sukmoonh@skku.edu}

\thanks{SH was supported by the National Research Foundation of Korea (NRF) grant funded by the Korean government (MSIT) (No.~RS-2023-00208874).}

\keywords{toric bundle, elementary transformation}

\subjclass[2020]{14M25, 14J60}

\begin{abstract}
We describe negative and positive elementary transformations of toric vector bundles in terms of Klyachko filtrations. We first determine the filtrations of the restriction of a toric vector bundle to an invariant prime divisor and then derive explicit formulas for the transformed bundle: only the filtration at the distinguished ray changes. As applications, we interpret adding and removing invariant boundary components of logarithmic tangent bundles as elementary transformations and recover the residue sequence for logarithmic cotangent bundles.
\end{abstract}

\maketitle

\section{Introduction}
Elementary transformations constitute one of the fundamental operations in the theory of vector bundles. They have long been used in the construction of moduli spaces, the study of stability, and various birational and logarithmic constructions; see \cite{HL, narashimhan-ramanan1}. Given a vector bundle $\Ee$ on a smooth variety $X$ together with a quotient bundle $\Qq$ supported on a divisor $D$, the elementary transformation modifies the bundle only along that divisor, while preserving its generic fiber. Although this construction is classical, it is useful to have a systematic description of it in terms of the combinatorial data of toric vector bundles. In this paper, we give such a description using Klyachko filtrations.

Toric vector bundles admit a remarkably simple combinatorial description through the Klyachko classification, which identifies equivariant vector bundles on a smooth toric variety with collections of compatible filtrations indexed by the rays of the defining fan. Since many geometric properties can be read directly from these filtrations, Klyachko's theory has become one of the principal tools in the study of toric vector bundles; \cite{Kly90, Kly91, Per03, KD19,DDK20, HH}. In this paper, we incorporate elementary transformations into this filtration-theoretic framework by describing their effect directly on Klyachko data. Although the relationship described here may be familiar to experts, we have not found a reference in which the effect of elementary transformations on Klyachko filtrations is formulated explicitly.

Our starting point is the restriction $\Ee_{|D}$ of a toric vector bundle to an invariant prime divisor $D$. Although the restriction is naturally equivariant, its generic fiber is no longer the original fiber but rather the associated graded space with respect to the distinguished filtration. We first describe this restriction completely by decomposing it into stabilizer-weight components and determining the induced Klyachko filtrations after passing to the effective quotient torus $\overline{T}$. The resulting formulas are intrinsic and depend only on intersections and quotient maps of the original filtrations. Building on this description, we obtain explicit formulas for elementary transformations. Suppose that
\[
q: \Ee_{|D}\twoheadrightarrow \Qq
\]
is an equivariant quotient bundle on an invariant divisor $D$. The quotient induces linear maps on the graded pieces of the distinguished filtration, and we prove that the Klyachko filtrations of the elementary transform are obtained by replacing the filtration corresponding to $D$ with the inverse images of the kernels of these graded maps, while all remaining filtrations are unchanged. 

We also obtain the analogous description for positive elementary transformations, together with the corresponding exact sequences and determinant formulas.
The filtration formulas provide an effective combinatorial realization of elementary transformations. Instead of reconstructing equivariant bundles geometrically, one may perform the transformation directly on finite-dimensional linear-algebraic data. This makes the construction particularly suitable for explicit calculations and applications.

As an application, we reinterpret logarithmic tangent bundles on smooth toric varieties in terms of elementary transformations. We prove that adding an invariant boundary component is precisely a negative elementary transformation, whereas removing a boundary component is realized by the corresponding positive transformation. Via duality, the same construction recovers the Poincar\'e residue sequence for logarithmic cotangent bundles. Thus elementary transformations provide a natural and uniform framework for understanding logarithmic bundles in the toric setting.

The final section presents several examples illustrating the general theory. Besides direct sums of equivariant line bundles, we study elementary transformations of tangent bundles and their exterior powers on projective spaces, showing that the filtration formulas immediately recover the transformed bundles. These examples demonstrate that the filtration-theoretic description is both conceptually natural and computationally effective.

We expect that the explicit formulas established here will be useful in further investigations of toric vector bundles. In particular, they provide a convenient tool for constructing new equivariant bundles, studying their stability, and analyzing moduli-theoretic questions through Klyachko's linear-algebraic framework.

\section{Preliminaries}
An $n$-dimensional  \textit{toric variety} is a normal irreducible variety $X$ containing a torus $T \simeq \left(\mathbb{C}^{\times}\right)^{n}$ as a Zariski open subset such that the action of $T$ on itself extends to an algebraic action of $T$ on $X$. Given $m = (a_{1}, \dots , a_{n}) \in \mathbb{Z}^n$, the character $\chi^{m}: (\mathbb{C}^{\times})^{n}\rightarrow \mathbb{C}^{\times}$ is defined by 
\[
\chi^{m}(t_{1},\dots,t_{n}) = {t_{1}}^{a_1} \dots {t_{n}}^{a_{n}}.
\]
The set of characters forms a free abelian group $M$ of rank $n$. A \textit{one-parameter subgroup} of $T$ is a morphism $\ u: \mathbb{C}^{\times} \rightarrow T$ that is a group homomorphism. The one-parameter subgroups form a free abelian group $N$ of rank $n$. We define $N_{\mathbb{R}} := N \otimes_{\mathbb{Z}} \mathbb{R}$ and $M_{\mathbb{R}} := M \otimes_{\mathbb{Z}} \mathbb{R}$. There is a natural pairing \( \langle \cdot, \cdot \rangle \colon M \times N \to \mathbb{Z} \), defined by the relation
\[
m \circ u(t) = t^{\langle m, u \rangle},
\]
for all \( t \in \mathbb{C}^\times \), where \( m \in M \) and \( u \in N \). Given a fan $\Sigma$ in $N_{\mathbb{R}}$, each cone $\sigma\in \Sigma$ produces an affine toric variety $U_{\sigma}= \operatorname{Spec}(\mathbb{C}[S_{\sigma}])$ for the semigroup $S_{\sigma}=\sigma^\vee \cap M$, gluing together to form a toric variety $X$ with $U_\sigma$ as a \( T \)-invariant affine chart of $X$; see \cite{Ful93} and \cite[Chapter 3]{CLS11}. There exists a one-to-one correspondence between the set of rays $\rho \in \Sigma(1)$ and the set of $T$-invariant Weil divisors $D_\rho$. Throughout the remainder of this paper, we always assume that $X=X_\Sigma$ is smooth and complete. For a fixed ray
\[
\rho_0=\RR_{\geq0}u_0\in\Sigma(1),
\]
with $u_0$ primitive, write $D_0=D_{\rho_0}$ the corresponding $T$-invariant prime divisor. The torus acting effectively on $D_0$ is the quotient $\overline T=T/T_0$, where $T_0\cong \mathbb{G}_m$ is the one-parameter subtorus determined by $u_0$. Its lattice of one-parameter subgroups is $\overline N=N/\ZZ u_0$ with the natural quotient map $p:N\rightarrow \overline N$. The fan of $D_0$ is
\begin{equation}
\Sigma_0=\mathrm{Star}_\Sigma(\rho_0)/\rho_0
=\left\{p(\sigma)\;\middle|\;\rho_0\preceq\sigma\in\Sigma\right\},
\end{equation}
and we have $D_0\cong  X_{\Sigma_0}$. Since $X$ is smooth, every ray $\bar\rho\in\Sigma_0(1)$ comes from a two-dimensional cone $\tau=\langle\rho_0,\rho\rangle\in\Sigma(2)$. Then the class $\bar u_\rho=p(u_\rho)$ is the primitive generator of $\bar\rho$. 

Denote the set of rays by $\Sigma(1)=\{\rho_0, \dots, \rho_q\}$, and let $u_i$ denote the primitive vector of $\rho_i$. For a fixed vector space $E$, we consider a collection of decreasing $\mathbb{Z}$-filtrations $\{E^{\rho_i}(\bullet) \}_{{\rho_i} \in \Sigma(1)}$ with 
\begin{equation}\label{Klyachko}
E^{\rho_i}(\bullet): \quad E\supseteq \dots \supseteq E^{\rho_i}(-1) \supseteq E^{\rho_i}(0) \supseteq E^{\rho_i} (1) \dots \supseteq 0
\end{equation}
for each ray $\rho_i\in \Sigma(1)$. The filtrations
$\{E^{\rho_i}(\bullet)\}_{\rho_i\in\Sigma(1)}$
are said to be {\it full} if, for each $\rho_i\in\Sigma(1)$,
$E^{\rho_i}(j)=E$ for $j\ll0$ and $E^{\rho_i}(j)=0$ for $j\gg0$.
Unless otherwise specified, all filtrations are assumed to be full. Then Klyachko presented a theorem to express $T$-equivariant reflexive sheaves as a set of decreasing $\ZZ$-filtrations.

\begin{theorem}\textnormal{\cite[Theorem 1.3.2]{Kly91}}\label{kly} The category of $T$-equivariant reflexive sheaves on $X$ is equivalent to the category of decreasing $\mathbb{Z}$-filtrations $\{E^{\rho_i} (\bullet)\}_{{\rho_i} \in \Sigma(1)}$. Such a reflexive sheaf is a toric bundle if and only if the corresponding filtrations satisfy the following compatibility condition:

\leftskip=2em 
\noindent $\mathrm{(C)}$ For any cone $\sigma \in \Sigma$, the filtrations $\{E^{\rho_i} (\bullet)\}_{{\rho_i} \in \sigma(1)}$, consist of coordinate subspaces of some basis of $E$. 
\rightskip=3em
\end{theorem}

\begin{example} \label{Tangentsheaf} \cite[Example 2.1.3]{KD19} 
The filtrations associated to the tangent bundle \( \mathcal{T}_X \) are given by:
\[
E^{\mathbf{\rho}_{i}}(j)=
\begin{cases}
\hspace{.2cm} E & \text{for } \quad j \leq 0 ;\\
\hspace{.2cm} \CC u_i & \text{for } \quad j = 1; \\
\hspace{.2cm} 0 & \text{for } \quad j > 1.
\end{cases}
\]
\end{example}

\noindent Let $\Ee$ be a toric vector bundle on $X$ with generic fiber $E$, associated to the Klyachko filtration \eqref{Klyachko}. The generic fiber of $\Ee_{|D_0}$ at $\eta_0 \in D_0$ is not naturally the original vector space $E$. At the generic point of $D_0$, a local equation for $D_0$ defines a valuation, and reduction of the corresponding lattice gives
\begin{equation}
\bigl(\Ee_{|D_0}\bigr)_{\eta_0} \cong \gr^{\rho_0}(E)\otimes_{\CC} \CC(D_0)
:= \bigoplus_{a\in\ZZ}\gr_a^{\rho_0}(E) \qquad \text { with }\gr_a^{\rho_0}(E)=
\frac{E^{\rho_0}(a)}{E^{\rho_0}(a+1)}.
\end{equation}
We now describe the Klyachko data of the restriction $\Ee_{|D_0}$. Let $\bar\rho\in\Sigma_0(1)$, and let $\tau=\langle\rho_0,\rho\rangle\in\Sigma(2)$ be the corresponding two-dimensional cone. For every $a\in\ZZ$, denote the quotient map by
\[
\pi_a:E^{\rho_0}(a)\longrightarrow\gr_a^{\rho_0}(E)
\]
and choose $m_0\in M$ satisfying $\langle m_0,u_0\rangle=1$. The next proposition gives both the unnormalized $T$-equivariant description and the normalized Klyachko filtrations for the effective quotient torus $\overline T=T/T_0$.

\begin{proposition}\label{restriction}
The following statements hold.
\begin{enumerate}
\item [(i)] The $T$-equivariant bundle $\Ee_{|D_0}$ decomposes canonically according to the action of the stabilizer $T_0$:
\[
\Ee_{|D_0}=\bigoplus_{a\in\ZZ}\Ee_{0,a},
\]
and the generic fiber of $\Ee_{0,a}$ is $\gr_a^{\rho_0}(E)\otimes_\CC \CC(D_0)$. The subgroup $T_0$ acts on this summand
through its character of weight $a$.

\item [(ii)] Let $\bar\rho\in\Sigma_0(1)$ correspond to the two-dimensional cone $\langle\rho_0,\rho\rangle$. Before removing the stabilizer character, the filtration on $\gr_a^{\rho_0}(E)$ belonging to $\bar\rho$ is
\begin{equation}\label{filt-factor}
\widetilde F_a^{\bar\rho}(j) =\pi_a\bigl(E^{\rho_0}(a)\cap E^\rho(j)\bigr).
\end{equation}

\item [(iii)] Choose $m_0\in M$ with $\langle m_0,u_0\rangle=1$. After twisting $\Ee_{0,a}$ by the character $-am_0$, the action of $T_0$ becomes trivial and the bundle descends to the effective torus $\overline T$. Its ordinary Klyachko filtration for the fan $\Sigma_0$ is
\begin{equation}\label{Klyachko-filt}
F_a^{\bar\rho}(j)=\pi_a\left(E^{\rho_0}(a)\cap E^\rho\bigl(j+a\langle m_0,u_\rho\rangle\bigr) \right).
\end{equation}

\item [(iv)] These filtrations satisfy Klyachko's compatibility condition on every cone of $\Sigma_0$. Together with the stabilizer weights $a$, they determine $\Ee_{|D_0}$ uniquely up to $T$-equivariant isomorphism.
\end{enumerate}
\end{proposition}

\begin{proof}
Let $\bar\sigma\in\Sigma_0$ and let $\sigma\in\Sigma$ be the corresponding
cone containing $\rho_0$, so that $\bar\sigma=\sigma/\rho_0$. Klyachko compatibility for $\Ee$ gives a finite
direct-sum decomposition
\begin{equation}\label{weight-decomp}
E=\bigoplus_{m\in M}E_m^\sigma
\end{equation}
such that, for every ray $\rho'\preceq\sigma$,
\begin{equation}\label{klyachko-coord-filt}
E^{\rho'}(j)=
\bigoplus_{\langle m,u_{\rho'}\rangle\geq j}E_m^\sigma.
\end{equation}
Only finitely many summands in \eqref{weight-decomp} are nonzero. The sign in \eqref{klyachko-coord-filt} is the sign corresponding to our decreasing-filtration convention. Taking $\rho'=\rho_0$ in \eqref{klyachko-coord-filt}, we obtain
\[
E^{\rho_0}(a)=\bigoplus_{\langle m,u_0\rangle\geq a}E_m^\sigma.
\]
Therefore, we have
\begin{equation}\label{graded-weight-decomp}
\gr_a^{\rho_0}(E) \cong \bigoplus_{\langle m,u_0\rangle=a}E_m^\sigma,
\end{equation}
discarding all weights except those lying on the affine hyperplane $\langle m,u_0\rangle=a$. The one-parameter subgroup $T_0$ acts on every summand in \eqref{graded-weight-decomp} with the same weight $a$. This proves the asserted description of the generic fiber in part~\textup{(i)}. Since a torus is diagonalizable and $T_0$ acts trivially on the base $D_0$, its eigenspaces glue to the subbundles $\Ee_{0,a}$.

For part~\textup{(ii)} observe first that Intersecting the two coordinate
filtrations gives
\[
E^{\rho_0}(a)\cap E^\rho(j)= \bigoplus_{\substack{\langle m,u_0\rangle\geq a\\\langle m,u_\rho\rangle\geq j}}E_m^\sigma.
\]
After quotienting by $E^{\rho_0}(a+1)$, only the terms with
$\langle m,u_0\rangle=a$ survive. Hence we have
\begin{equation}\label{induced-filt}
\pi_a\bigl(E^{\rho_0}(a)\cap E^\rho(j)\bigr)=\bigoplus_{\substack{\langle m,u_0\rangle=a\\\langle m,u_\rho\rangle\geq j}}E_m^\sigma.
\end{equation}
This is exactly the filtration of the component of weight $a$ by the remaining $v_\rho$-weights, proving part~\textup{(ii)}. Although we choose a compatible weight decomposition, the resulting filtration is intrinsic: it is the image of the canonical intersection $E^{\rho_0}(a)\cap E^\rho(j)$ under the canonical quotient map $E^{\rho_0}(a)\to\operatorname{gr}^{\rho_0}_a(E)$. Hence it is independent of the chosen weight decomposition.

The characters $m$ occurring in \eqref{graded-weight-decomp} do not themselves belong to the character lattice of $\overline T$, because they have $\langle m,u_0\rangle=a$ rather than zero. Choosing $m_0\in M$ with
$\langle m_0,u_0\rangle=1$, 
\[
\bar m:=m-am_0
\]
annihilates $u_0$ and therefore defines a character of $\overline T$. Its
pairing with the ray $\bar u_\rho$ is
\[ 
\langle\bar m,\bar u_\rho\rangle=\langle m,u_\rho\rangle-a\langle m_0,u_\rho\rangle.
\]
The inequality $\langle\bar m,\bar u_\rho\rangle\geq j$ is therefore equivalent to
\[
\langle m,u_\rho\rangle \geq j+a\langle m_0,u_\rho\rangle.
\]
Substitution into \eqref{induced-filt} gives precisely
the formula in part~\textup{(iii)}.

For the cone $\bar\sigma$, the spaces $E_m^\sigma$ with
$\langle m,u_0\rangle=a$, relabeled by the quotient characters
$\bar m=m-am_0$, provide a simultaneous weight decomposition of
$\gr_a^{\rho_0}(E)$ for all rays of $\bar\sigma$. Hence the filtrations in
part~\textup{(iii)} satisfy Klyachko compatibility on $\bar\sigma$. Repeating
the construction for every cone containing $\rho_0$ gives compatible local
descriptions. On overlaps, they yield the same intrinsic filtrations because
those filtrations are defined by intersections and quotient maps, not by a
chosen splitting. Klyachko's classification on $X_{\Sigma_0}=D_0$ now
reconstructs each normalized bundle $\Ee_{0,a}$ uniquely. Reattaching the
$T_0$-character of weight $a$ and taking the direct sum reconstructs
$\Ee_{|D_0}$, proving part~\textup{(iv)}.
\end{proof}


\section{Elementary transformation}

Suppose that there is a surjective equivariant morphism on $D_0$:
\begin{equation}\label{quotient}
q:~~\Ee_{|D_0}\twoheadrightarrow\Qq,
\end{equation}
where $\Ee$ is locally free as an $\Oo_X$-module of rank $s$ and $\Qq$ is locally free as an $\Oo_{D_0}$-module of rank $r$. Because $T_0$ acts trivially on $D_0$ and is diagonalizable, both vector bundles decompose canonically into stabilizer-weight subbundles:
\[
\Ee_{|D_0}=\bigoplus_{a\in\ZZ}\Ee_{0,a}, \qquad \Qq=\bigoplus_{a\in\ZZ}\Qq_a
\]
as in Proposition \ref{restriction}. Define $W_a$ by
\[
(\Qq_a)_{\eta_0}=W_a\otimes_\CC \CC(D_0).
\]
Thus $\Ee_{0,a}$ and $\Qq_a$ are vector bundles on $D_0$, whereas $\gr_a^{\rho_0}(E)$ and $W_a$ denote their generic vector spaces after suppressing extension of scalars to $\CC(D_0)$. Since $q$ is $T_0$-equivariant, it preserves stabilizer weights and decomposes as a direct sum of surjective bundle maps
\begin{equation}\label{stab-weight-map}
q=\bigoplus_a q_a, \qquad q_a:\Ee_{0,a}\twoheadrightarrow\Qq_a.
\end{equation}
On generic fibers, using the preceding identifications and suppressing the
factor $k(D_0)$, these maps are the linear surjections
\begin{equation}
q_a:\gr_a^{\rho_0}(E)\twoheadrightarrow W_a.
\end{equation}
With the same character normalization as in \eqref{Klyachko-filt}, let
\[
Q_a^{\bar\rho}(j)\subseteq W_a
\]
be its Klyachko filtration at $\bar\rho\in\Sigma_0(1)$. Since an equivariant bundle map is a morphism of Klyachko-filtered vector spaces, we have
\[
q_a\bigl(F_a^{\bar\rho}(j)\bigr)\subseteq Q_a^{\bar\rho}(j)
\]
for all $a,\bar\rho,j$. In fact, $q_a$ is strict, i.e., this inclusion is equality. Fix $\bar\rho\in\Sigma_0(1)$, let $R_{\bar\rho}:=\Oo_{D_0,\eta_{\bar\rho}}$ be the discrete valuation ring at the generic point $\eta_{\bar \rho}$, and choose a uniformizer $t$. Let $L_{E,a}$ and $L_{Q,a}$ be the equivariant $R_{\bar\rho}$-lattices determined by $\Ee_{0,a}$ and $\Qq_a$. Since $q_a:\Ee_{0,a}\rightarrow \Qq_a$ is a surjective morphism of
vector bundles, localization gives a surjection $L_{E,a}\twoheadrightarrow L_{Q,a}$. Under the lattice--filtration correspondence, we have
\[
F_a^{\bar\rho}(j) =\left\{e\mid t^{-j}e\in L_{E,a}\right\}, \qquad
Q_a^{\bar\rho}(j)=\left\{y\mid t^{-j}y\in L_{Q,a}\right\}.
\]
For $y\in Q_a^{\bar\rho}(j)$, we have $t^{-j}y\in L_{Q,a}$ so that lattice
surjectivity provides $\ell\in L_{E,a}$ with $q_a(\ell)=t^{-j}y$. Setting $e=t^j\ell$, we then have $t^{-j}e=\ell\in L_{E,a}$ and
$q_a(e)=y$. Thus we have
\[
Q_a^{\bar\rho}(j)\subseteq q_a(F_a^{\bar\rho}(j)).
\]
Together with the filtration-preserving inclusion above, this proves that
$q_a$ is \emph{strict}, namely
\begin{equation}
\label{eq:strict-quotient-filtration}
Q_a^{\bar\rho}(j)
=q_a\bigl(F_a^{\bar\rho}(j)\bigr).
\end{equation}
Combining this with \eqref{Klyachko-filt} gives
\[
Q_a^{\bar\rho}(j)=q_a\!\left(\pi_a\!\left(E^{\rho_0}(a)\cap E^\rho\bigl(j+a\langle m_0,u_\rho\rangle\bigr)\right)\right).
\]
If the stabilizer character is left unnormalized, the simpler form would be
\[
\widetilde Q_a^{\bar\rho}(j)=q_a\!\left(\frac{E^{\rho_0}(a)\cap E^\rho(j)+E^{\rho_0}(a+1)}{E^{\rho_0}(a+1)}\right).
\]
Consequently, the maps $q_a$ used in an elementary transform cannot be chosen independently of the other filtrations of $\Ee$. They must be compatible with every ray of $\Sigma_0$, equivalently with every two-dimensional cone of $\Sigma$ containing $\rho_0$.

\begin{remark}
A collection of surjective linear maps $q_a:\gr_a^{\rho_0}(E)\to W_a$ determines the codimension-one change in the elementary transform, but it does not automatically determine a toric vector bundle quotient on all of $D_0$. The image filtrations in (3.6) must satisfy Klyachko compatibility on every cone of $\Sigma_0$.
\end{remark}

Let $i:D_0\hookrightarrow X$, and let $q:\Ee_{|D_0}\twoheadrightarrow\Qq$ be an equivariant quotient in which $\Qq$ is locally free of rank $r$ on $D_0$. Composing restriction with $q$ gives
\begin{equation}\label{elem-transSeq}
\Ee \longrightarrow i_*\Ee_{|D_0}\longrightarrow i_*\Qq.
\end{equation}

\begin{proposition}\label{elem-transf}
The (negative) elementary transformation of $\Ee$ with respect to \eqref{quotient}
\begin{equation}
\Ee^-:=\mathrm{Elm}_{D_0,\Qq}(\Ee) =\ker\bigl(\Ee\rightarrow i_*\Qq\bigr)
\end{equation}
is a $T$-equivariant vector bundle of $\mathrm{rank}(\Ee^-)=\mathrm{rank}(\Ee)$. Its generic fiber on the dense torus is still $E$, and its Klyachko filtrations are given as follows:
\begin{equation}\label{elem-trans-ray}
\begin{cases}
\bigl(E^-\bigr)^\rho(j)=E^\rho(j), \qquad \rho\neq\rho_0;\\[2pt]
\bigl(E^-\bigr)^{\rho_0}(a)=\pi_a^{-1}(\ker q_a)=
\left\{e\in E^{\rho_0}(a)\;\middle|\;q_a\bigl(\overline{e}\bigr)=0\right\},
\end{cases}
\end{equation}
where $\overline{e}\in \gr_a^{\rho_0}(E)$ is the corresponding class to $e$. In particular, we have $\det (\Ee^-)\cong \det (\Ee)\otimes \Oo_X(-rD_0)$. 
\end{proposition}

\begin{proof}
Since the map $\Ee \rightarrow i_*\Qq$ is the composite of two $T$-equivariant maps, $\Ee^-$ is $T$-equivariant vector bundle. The dense torus $T\subset X$ is disjoint from every invariant prime divisor,
in particular from $D_0$. Restricting \eqref{elem-transSeq} to $T$ therefore gives
\[
\bigl(\Ee^-\bigr)_{|T}\cong \Ee_{|T},
\]
and so two bundles consequently have the same generic fiber $E$. Now for $\rho\neq\rho_0$ the generic point $\eta_\rho$ of $D_\rho$ does not lie on $D_0$, because $D_\rho$ and $D_0$ are distinct prime divisors. Thus we have $(i_*\Qq)_{\eta_\rho}\cong 0$ and $\bigl(\Ee^-\bigr)_{\eta_\rho}\cong\Ee_{\eta_\rho}$. The Klyachko filtration at $\rho$ is determined by this equivariant lattice over the discrete valuation ring $\Oo_{X,\eta_\rho}$. The lattices are equal, so the filtrations are equal. 

Let $R_0=\Oo_{X,\eta_0}$ be the discrete valuation ring with $t_0$ a uniformizer, and $L_0$ be the $R_0$-lattice determined by $\Ee_{\eta_0}$. Reduction modulo $t_0$ identifies the fiber of $\Ee_{|D_0}$ with
\[
L_0/t_0L_0 \cong \gr^{\rho_0}(E)\otimes_\CC \CC(D_0)=\bigoplus_a\gr_a^{\rho_0}(E)\otimes_\CC \CC(D_0).
\]
The stabilizer $T_0$ is diagonalizable, so the quotient map decomposes by
weights as $q=\bigoplus_a q_a$. Then the lattice of the elementary transform is the inverse image of the kernel of
the reduction map:
\begin{equation}\label{elem-trans-lattice}
L_0' =\left\{\ell\in L_0\;\middle|\;q(\overline{\ell})=0\right\}=\mathrm{red}^{-1}(\ker q),
\end{equation}
where $\overline{\ell}\in L_0/t_0L_0$ is the class corresponding to $\ell$. Under the lattice--filtration correspondence, the reduction of the
component of weght $a$ is the canonical quotient map
\[
\pi_a:E^{\rho_0}(a)\longrightarrow E^{\rho_0}(a)/E^{\rho_0}(a+1)=\gr_a^{\rho_0}(E).
\]
Therefore an element $e\in E^{\rho_0}(a)$ lies in the component of weight $a$ of $L_0'$ exactly when $q_a\bigl(\pi_a(e)\bigr)=0$. Equivalently,
\[
e\in\pi_a^{-1}(\ker q_a).
\]
Notice that the formula uses only the generic maps $q_a$. Their compatibility with the filtrations along the rays of $\Sigma_0$ is what ensures that they arise from the global equivariant quotient map $q$. Once that quotient exists, Klyachko compatibility for the transformed filtrations also follows from the fact that $\Ee'$ is an equivariant vector bundle.

Finally, the assertion on $\det (\Ee^-)$ follows from the standard computation. Indeed, the inclusion $\Ee^-\hookrightarrow\Ee$ is locally represented by
\[
\mathrm{diag}(1,\ldots,1, \underbrace{t,\ldots,t}_{r\text{ times}}).
\]
Its determinant vanishes to order $r$ along $D_0$ and nowhere else. 
\end{proof}

\begin{example}
Let $X=\mathbb P^2$ with rays generated by $u_0=(1,0), u_1=(0,1),u_2=(-1,-1)$. Put $D=D_{\rho_0}$ and take $\Ee=\Oo_{\PP^2}^{\oplus 2}$ to consider $q: \Ee_{|D}\twoheadrightarrow\Oo_D$ defined by $ae_1+be_2\mapsto b$. At each ray, we have
\[
E^\rho(i)=\begin{cases}E=\CC e_1 \oplus \CC e_2 ,&i\leq0;\\ 0,&i\geq1.
\end{cases}
\]
At $\rho_0$, only $\gr_0^{\rho_0}(E)=E$ is nonzero and
$q_0(ae_1+be_2)=b$. Thus we have
\begin{equation}
\bigl(E^-\bigr)^{\rho_0}(i)=\begin{cases}E,&i\leq-1;\\ \CC e_1,&i=0;\\0,&i\geq1.
\end{cases}
\end{equation}
The other two ray filtrations remain unchanged, and so we obtain
\[
\Ee^- \cong \Oo_{\PP^2}\oplus \Oo_{\PP^2}(-D_0)\cong \Oo_{\PP^2}\oplus \Oo_{\PP^2}(-1). 
\]
\end{example}

\begin{example}
Let $X=\mathbb F_a$ be the Hirzebruch surface with ray generators $u_0=(0,1), u_1=(1,0), u_2=(0,-1), u_3=(-1,a)$. The corresponding $T$-invariant divisors are $D_0,\ldots,D_3$ and consider the projection $\Ee_{|D}\twoheadrightarrow\Oo_{\FF_a}(D_1+D_2)_{|D}$ for $D=D_2$ and $\Ee=\Oo_{\FF_a}(D_1+D_2)\oplus\Oo_{\FF_a}(D_3)$. Fix a basis $\{ e_1, e_2\}$ for $E$ from the each factor of $\Ee$. Then in the filtration of $\rho_2$, the $e_1$-line jumps at $1$, while the $e_2$-line jumps at $0$ so that we have
\[
E^{\rho_2}(i)=\begin{cases}
\CC e_1\oplus \CC e_2,&i\leq0;\\
\CC e_1,&i=1;\\
0,&i\geq2.
\end{cases}
\]
The only nonzero component of the map $q$ in \eqref{quotient} is the isomorphism from the $e_1$-part of $\gr_1^{\rho_2}(E)$. Consequently, we have 
\begin{equation}
\bigl(E^-\bigr)^{\rho_2}(i)=
\begin{cases}
\CC e_1\oplus \CC e_2,&i\leq0;\\
0,&i\geq1.
\end{cases}
\end{equation}
All other ray filtrations are unchanged, so we have
\[
\Ee^-\cong \Oo_{\FF_a}(D_1)\oplus\Oo_{\FF_a}(D_3).
\]
Thus the transform lowers the coefficient of $D_2$ in the first line-bundle summand $\Oo_{\FF_a}(D_1+D_2)$ by one.
\end{example}

\begin{example}
For any smooth toric variety $X$, fix $D=D_{\rho_0}$ and take the identity quotient $q:~\Ee_{|D}\xrightarrow{\cong}\Ee_{|D}$. Since each $q_a$ is the identity on $\gr_a^{\rho_0}(E)$, so we have
\begin{equation}
\bigl(E^-\bigr)^{\rho_0}(i)=E^{\rho_0}(i+1)  \qquad\text{and} \qquad 
\bigl(E^-\bigr)^\rho(i)=E^\rho(i)~~\text{ for }\rho\neq\rho_0.
\end{equation}
Therefore we have $\Ee^-\cong \Ee(-D)$.
\end{example}

\noindent Retain the equivariant quotient map \eqref{quotient} and set $\Kk:=\ker(q)$ a vector bundle of rank $s-r$ on $D_0$. 
\begin{definition}
The \emph{positive elementary transformation} associated with $q$ is defined by
\begin{equation}\label{positive-transform}
\mathrm{Elm}^+_{D_0,q}(\Ee):=\mathrm{Elm}_{D_0,\Qq}(\Ee)\otimes\Oo_X(D_0).
\end{equation}
Thus a negative transformation first forces the restriction to lie in $\Kk$,
whereas the subsequent twist by $D_0$ allows a simple pole in precisely the
$\Kk$-directions.
\end{definition}

\begin{proposition}\label{positive-transform-filt}
The positive elementary transformation $\Ee^+:=\mathrm{Elm}^+_{D_0,q}(\Ee)$ is an $T$-equivariant vector bundle of rank $s$ , fitting into an exact sequence
\begin{equation}\label{positive-transform-exact}
0\to \Ee \to \Ee^+ \to  i_*\bigl(\Kk \otimes\Oo_{D_0}(D_0)\bigr) \to 0.
\end{equation}
Its generic fiber on the dense torus is still $E$, and its Klyachko filtrations are given as follows:
\begin{equation}\label{elem-trans-ray-pos}
\begin{cases}
\bigl(E^+\bigr)^\rho(j)=E^\rho(j), \qquad \rho\neq\rho_0;\\[2pt]
\bigl(E^+\bigr)^{\rho_0}(a)=\pi_{a-1}^{-1}(\ker q_{a-1}),
\end{cases}
\end{equation}
where $\pi_{a-1}: E^{\rho_0}(a-1)\rightarrow \gr_{a-1}^{\rho_0}(E)$. In particular, we have $\det (\Ee^+)\cong \det (\Ee)\otimes \Oo_X((s-r)D_0)$. 
\end{proposition}

\begin{proof}
By definition, $\Ee^+=\Ee^-(D_0)$ with $\Ee^-=\mathrm{Elm}_{D_0,\Qq}(\Ee)$ is locally free and $T$-equivariant. Set 
\[
K_a=\ker (q_a)
\]
in \eqref{stab-weight-map}. Proposition~\ref{elem-transf} gives, for every integer $b$, $(E^-)^{\rho_0}(b)=\pi_b^{-1}(K_b)$. Taking $b=a-1$, we obtain
\[
(E^+)^{\rho_0}(a)=(E^-)^{\rho_0}(a-1)=\pi_{a-1}^{-1}(K_{a-1})=\left\{
e\in E^{\rho_0}(a-1)
\ \middle|\
q_{a-1}\bigl(\overline{e}\bigr)=0
\right\}.
\]
Thus we have
\begin{equation}
(E^+)^{\rho_0}(a)=\pi_{a-1}^{-1}(K_{a-1})\subset E^{\rho_0}(a-1)
\end{equation}
directly from the negative transformation filtration. It also explains the index $a-1$: the positive transform first imposes the kernel condition on the $(a-1)$-st graded piece and then the twist by $D_0$ moves that contribution to filtration index $a$. Equivalently, $(E^+)^{\rho_0}(a)$ is the unique subspace of $E^{\rho_0}(a-1)$ that contains $E^{\rho_0}(a)$ and whose image in
\begin{equation}
E^{\rho_0}(a-1)/E^{\rho_0}(a)
=\gr_{a-1}^{\rho_0}(E)
\end{equation}
is exactly $K_{a-1}$. The negative transformation has the same generic fiber $E$ and the same filtrations as $\Ee$ at every $\rho\neq\rho_0$. Therefore we have
\begin{equation}
(E^+)^\rho(a)=E^\rho(a) \qquad \text{For }\rho\neq\rho_0.
\end{equation}

On the weight-$a$ graded piece of $\Ee$, the negative transformation leaves $K_a$ at weight $a$ and lowers the complementary quotient directions from weight $a$ to weight $a-1$. Tensoring by $\Oo_X(D_0)$ raises every weight at $\rho_0$ by one. Hence the positive transform raises the kernel directions $K_a$ from weight $a$ to weight $a+1$, while the quotient directions retain weight $a$. At index $a$, it therefore contains $E^{\rho_0}(a)$ together with exactly those classes in $\gr_{a-1}^{\rho_0}(E)$ that lie in $K_{a-1}$. Intrinsically, this is
\[
(E^+)^{\rho_0}(a)=\pi_{a-1}^{-1}(K_{a-1}).
\]
No choice of a complement to $K_{a-1}$ is involved in this inverse-image formula. The transformation is supported on $D_0$, so its generic fiber is still $E$ and its filtrations at $\rho\neq\rho_0$ are unchanged.

For every integer $a$, we have $\ker (\pi_{a-1})=E^{\rho_0}(a)$. Therefore
\begin{equation}
E^{\rho_0}(a)=\ker\pi_{a-1}\subseteq\pi_{a-1}^{-1}(K_{a-1})=(E^+)^{\rho_0}(a).
\end{equation}
At every other ray the two filtrations are equal. Thus the identity map on the common generic fiber $E$ preserves all Klyachko filtrations. By Klyachko's equivalence, it induces an equivariant morphism $\Ee \rightarrow \Ee^+$, which is the identity over the dense torus and so is injective.

As in the proof of Proposition \ref{elem-transf}, let $R_0=\Oo_{X,\eta_0}$ be the discrete valuation ring with $t_0$ a uniformizer, and $L_0$ be the $R_0$-lattice determined by $\Ee_{\eta_0}$. In each weight-$a$ component, the lattice of $\Ee^+$ is obtained from the lattice of $\Ee$ by adjoining $t_0^{-1}$ times the directions in $K_a$ and by making no change in the quotient directions. Consequently, we have
\begin{equation}
L_0^+/L_0 \cong \left(\bigoplus_a K_a\right) \otimes_\CC \left(t_0^{-1}R_0/R_0\right).
\end{equation}
Here, $\bigoplus_aK_a$ is the stabilizer-weight decomposition of the generic fiber of $\Kk=\ker q$ on $D_0$, while $\left(t_0^{-1}R_0/R_0\right)$ is the generic fiber of $\Oo_{D_0}(D_0)=\Nn_{D_0|X}$. Hence the cokernel at the generic point of $D_0$ is the generic fiber of $\Kk\otimes\Oo_{D_0}(D_0)$. This identification also respects the filtrations coming from every ray of
the quotient fan $\Sigma_0$. Indeed, the filtrations of $\Kk$ are obtained by taking the kernels of the strict filtered maps defining $q$, and the factor $\left(t_0^{-1}R_0/R_0\right)$ contributes precisely the equivariant filtration of the normal line $\Oo_{D_0}(D_0)$. Therefore the filtered cokernel is the Klyachko data of $\Kk\otimes\Oo_{D_0}(D_0)$. Klyachko's equivalence on $D_0$ now identifies the global cokernel with $i_*\bigl(\Kk\otimes\Oo_{D_0}(D_0)\bigr)$. 

Finally, Proposition~\ref{elem-transf} gives
\[
\det(\Ee^+)\cong \det(\Ee^-)\otimes\Oo_X(nD_0)\cong \det(\Ee)\otimes\Oo_X((s-r)D_0). \qedhere
\]
\end{proof}

\begin{remark}
Equivalently, one may start with an equivariant subbundle $\Kk \subseteq\Ee_{|D_0}$ and define the positive transform as the subsheaf of $\Ee(D_0)$ consisting of meromorphic sections with at most a simple pole along $D_0$ whose polar value belongs to $\Kk\otimes\Oo_{D_0}(D_0)$. Formula \eqref{positive-transform-exact} is the corresponding residue sequence. Thus the negative transform lowers the quotient directions by one, while the positive transform raises the kernel directions by one.
\end{remark}

\begin{remark}
For fixed toric vector bundles $\Ee_{|D_0}$ and $\Qq$, define
\[
\Hh_a:=
\left\{
f_a\in\Hom_\CC \bigl(\gr_a^{\rho_0}(E),W_a\bigr)
\ \middle|\
f_a\bigl(F_a^{\bar\rho}(j)\bigr)
\subseteq Q_a^{\bar\rho}(j)
\text{ for every }\bar\rho\in\Sigma_0(1),\ \forall j\in\ZZ
\right\}
\]
for every stabilizer weight $a$. Each inclusion is a system of linear equations in the entries of $f_a$, so $\Hh_a$ is a finite-dimensional vector space. Klyachko's description of morphisms, together with the stabilizer-weight decomposition, gives
\[
\Hom_T(\Ee_{|D_0},\Qq) \cong \bigoplus_a\Hh_a,
\]
inside which there exists a Zariski open subset
\[
\mathrm{Sur}_T(\Ee_{|D_0},\Qq):=
\left\{ q\in\Hom_T(\Ee_{|D_0},\Qq)\mid q\text{ is surjective as a bundle map}
\right\}.
\]
If it is nonempty, then
\[
\dim\mathrm{Sur}_T(\Ee_{|D_0},\Qq) =\sum_a\dim\Hh_a.
\]
Note that $q$ is surjective if and only if the maps
\[
q_x:(\Ee_{|D_0})_x\twoheadrightarrow \Qq_x
\]
is surjective for every $T$-fixed point $x\in D_0$. In particular, $\mathrm{Sur}_T(\Ee_{|D_0},\Qq)$ is defined as a quasi affine determinantal variety. Note also that two surjections with fixed target $\Qq$ determine the same kernel if and only if they differ by composition with an equivariant automorphism of $\Qq$. Therefore the parameter space of quotient maps with target isomorphic to $\Qq$, is represented by the quotient
\[
\mathrm{Sur}_T(\Ee_{|D_0},\Qq)/\mathrm{Aut}_T(\Qq).
\]
\end{remark}


\section{Logarithmic tangent bundle}

For a reduced divisor $B=\sum_{\rho\in S}D_{\rho}$ on a smooth toric variety $X$, one can define the logarithmic tangent bundle
\begin{equation}
\Tt_X(-\log B)
=\left\{\theta\in \Tt_X\;\middle|\;\theta(\Ii_{B, X})\subseteq \Ii_{B,X}\right\},
\end{equation}
which is roughly speaking the sheaf of vector fields that are tangent to every component of $B$. Fix a smooth divisor $D=D_0\notin S$ so that $B+D_0$ again has simple normal crossings. Restricting the exact sequence
\[
0\to \Tt_X\bigl(-\log (B+D_0)\bigr) \to \Tt_X(-\log B) \to i_*\Nn_{D_0|X} \to 0
\]
to $D_0$, one has a surjection $q: ~\Tt_X(-\log B)_{|D_0}\twoheadrightarrow \Nn_{D_0|X}$.

\begin{proposition}\label{log-tangent-filt}
We have
\begin{equation}
\mathrm{Elm}_{D_0,\Nn_{D_0/X}}\bigl(\Tt_X(-\log B)\bigr)=\Tt_X\bigl(-\log(B+D_0)\bigr),
\end{equation}
i.e., the elementary transform is precisely the subsheaf of logarithmic vector fields that are additionally tangent to $D_0$. 
\end{proposition}

\begin{proof}
Recall that filtrations for $\Tt_X$ from Example \ref{Tangentsheaf}. The line $\CC u_{\rho}$ is the normal direction to the invariant divisor $D_\rho$. Requiring a toric vector field to be tangent to $D_\rho$ removes this normal direction. Consequently, the Klyachko filtrations of $\Tt_X(-\log B)$ are
\begin{equation}\label{log-tangent-klyachko}
\Tt_X(-\log B)^\rho(j)=
\begin{cases}
E,&j\leq0;\\
0,&j\geq1
\end{cases}
\qquad \text{ for }\rho\in S,
\end{equation}
and is equal to the filtrations for $\Tt_X$ when $\rho\notin S$. Since $\rho_0\notin S$, the filtration of $\Tt_X(-\log B)$ at $\rho_0$ is
\begin{equation}\label{log-tangent-before}
\Tt_X(-\log B)^{\rho_0}(j)=
\begin{cases}
E,&j\leq0;\\
\CC u_0,&j=1;\\
0,&j\geq2.
\end{cases}
\end{equation}
Its only nonzero associated graded pieces are
\[
\gr_0^{\rho_0}\bigl(\Tt_X(-\log B)\bigr)=E/\CC u_0,
\qquad
\gr_1^{\rho_0}\bigl(\Tt_X(-\log B)\bigr)=\CC u_0.
\]
Then the induced maps on each graded piece are
\begin{equation}\label{normal-grade}
q_0:~E/\CC u_0\longrightarrow 0,
\qquad
q_1:~\CC u_0\xrightarrow{\cong }\bigl(\Nn_{D_0|X}\bigr)_{\eta_0}.
\end{equation}
Here, the first map is zero because the weight-zero component consists of directions tangent to $D_0$, and the second map is an isomorphism because $u_0$ is precisely the one-dimensional normal direction.

Let $\Ff:=\ker\left(\Tt_X(-\log B)\longrightarrow i_*\Nn_{D_0|X}\right)$. By Proposition~\ref{elem-transf}, all filtrations of $\Ff$ at rays $\rho\neq\rho_0$ agree with those of $\Tt_X(-\log B)$. At $\rho_0$, \eqref{elem-trans-ray} and \eqref{normal-grade} give
\[
\begin{cases}
F^{\rho_0}(0)&=\pi_0^{-1}(\ker q_0)=\pi_0^{-1}(E/\CC u_0)=E;\\
F^{\rho_0}(1)&=\pi_1^{-1}(\ker q_1)=\pi_1^{-1}(0)=0.
\end{cases}
\]
It follows that
\begin{equation}\label{reft}
F^{\rho_0}(j)=
\begin{cases}
E,&j\leq0;\\
0,&j\geq1.
\end{cases}
\end{equation}
Adding $D_0$ to $B$ changes the set of boundary rays from $S$ to $S\cup\{\rho_0\}$. Then the filtrations of
$\Tt_X(-\log(B+D_0))$ agree with those of $\Tt_X(-\log B)$ at $\rho\neq\rho_0$, while its filtration at $\rho_0$ is exactly \eqref{reft}. Thus $\Ff$ and $\Tt_X(-\log(B+D_0))$ have identical Klyachko data, concluding that $\Ff \cong \Tt_X(-\log(B+D_0))$.
\end{proof}

\begin{remark}
The proof of Proposition~\ref{log-tangent-filt} shows that adding one invariant boundary component deletes the one-dimensional normal step at its ray and changes no other ray filtration. The change at $\rho_0$ has the effect to annihilate the line $\CC u_0$ for $j=1$. This description makes it possible to iterate the construction along any collection of invariant divisors.
\end{remark}

\begin{example}\label{full-div}
Let $\partial X=\sum_{\rho\in\Sigma(1)}D_\rho$ be the full toric boundary of a smooth toric variety $X$. Starting with $\Tt_X$ and successively applying the rank-one elementary transform along every invariant divisor gives
\begin{equation}
\Tt_X(-\log\partial X)\cong  N\otimes_\ZZ \Oo_X;
\end{equation}
see \cite[page 87]{Ful93}. Indeed, after all transformations, every ray filtration is
\[
E \text{ for }j\leq0,\quad \text{ and } \quad 0 \text{ for } j\geq1,
\]
which is the Klyachko data of the trivial bundle with fiber $E$. Geometrically, the infinitesimal $T$-action trivializes the logarithmic tangent bundle of the toric pair $(X,\partial X)$.
\end{example}

Conversely, consider the logarithmic tangent bundle $\Ee\cong\Tt_X\bigl(-\log(B+D_0)\bigr)$. Its restriction to $D_0$ fits into an exact sequence
\begin{equation}\label{restrict}
0\to \Oo_{D_0} \to \Tt_X\bigl(-\log(B+D_0)\bigr)_{|D_0}
\xrightarrow{~q~} \Tt_{D_0}(-\log B_0) \to 0
\end{equation}
with $B_0:=B_{|D_0}$. The filtration for $\Ee$ at $\rho_0$ is $E^{\rho_0}(a)=
\begin{cases}
E,&a\leq0;\\
0,&a\geq1,
\end{cases}$ and so its only nonzero associated graded piece is $\gr_0^{\rho_0}(E)=E$. The generic fiber of $T_{D_0}(-\log B_0)$ is $\overline E=E/\CC u_0$, and the map induced by $q$ on the only nonzero graded piece is
\begin{equation}
q_0:E\longrightarrow\overline E,
\qquad
\ker q_0=\CC u_0.
\end{equation}
In particular, the trivial line in \eqref{restrict} is represented in Klyachko data by the normal ray line $\CC u_0\subset E$. Apply the positive transformation to the quotient $q$ in \eqref{restrict} to have
\begin{equation}
0\to  \Tt_X\bigl(-\log(B+D_0)\bigr)\to \Ee^+:=\mathrm{Elm}^+_{D_0,q}\!\left(\Tt_X\bigl(-\log(B+D_0)\bigr)\right) \to i_*\Nn_{D_0|X} \to 0
\end{equation}
by Proposition~\ref{positive-transform-filt}, because we have
\[
\ker q\simeq\Oo_{D_0} \quad \text{and} \quad \ker q\otimes\Oo_{D_0}(D_0)
\cong \Nn_{D_0|X},
\]

\begin{proposition}
We have
\begin{equation}\label{eq:positive-transform-log-tangent}
\mathrm{Elm}^+_{D_0,q}\!\left(\Tt_X\bigl(-\log(B+D_0)\bigr)\right)
\cong \Tt_X(-\log B),
\end{equation}
i.e., the positive elementary transformation removes the tangency condition along $D_0$.
\end{proposition}

\begin{proof}
The filtration for $\Ee^+$ at $\rho_0$ is computed by Proposition~\ref{positive-transform-filt} as
\begin{equation}
(E^+)^{\rho_0}(a)
=\pi_{a-1}^{-1}(\ker q_{a-1}).
\end{equation}
Together with the filtration $\{E^{\rho}(a)\}_{a\in \ZZ}$, we obtain
\begin{equation}
(E^+)^{\rho_0}(a)=
\begin{cases}
E,&a\leq0;\\
\CC u_0,&a=1;\\
0,&a\geq2.
\end{cases}
\end{equation}
Indeed, the new step at index $1$ is
\[
(E^+)^{\rho_0}(1)
=\pi_0^{-1}(\ker q_0)
=\CC v_0.
\]
At every other ray $\rho\neq\rho_0$, the positive transformation leaves the filtration unchanged so that we have the Klyachko data of $\Tt_X(-\log B)$.  
\end{proof}

There is also a dual interpretation for logarithmic cotangent bundles. Dualizing the logarithmic tangent picture leads to the Poincar\'e residue sequence
\begin{equation}
0\to \Omega_X^1(\log B) \to \Omega_X^1\bigl(\log(B+D_0)\bigr) \to i_*\Oo_{D_0} \to 0.
\end{equation}
For $\Ee=\Omega_X^1(\log B)$, the logarithmic conormal sequence on $D_0$ is
\begin{equation}\label{restr-log-cot-pos-quot}
0\to \Nn_{D_0|X}^{\vee}\to \Omega_X^1(\log B)_{|D_0}\xrightarrow{~q~} \Omega_{D_0}^1(\log B_0) \to 0.
\end{equation}
Consequently, the kernel in the positive transformation construction is $\Kk=\Nn_{D_0|X}^{\vee}$ so that $\Kk\otimes\Oo_{D_0}(D_0)\cong \Oo_{D_0}$. Proposition~\ref{positive-transform-filt} therefore produces an exact sequence
\begin{equation}
0\to \Omega_X^1(\log B)\to \mathrm{Elm}^+_{D_0,q}\bigl(\Omega_X^1(\log B)\bigr) \to  i_*\Oo_{D_0} \to 0.
\end{equation}
\begin{corollary}
\begin{equation}
\label{pos-transf-log-cot}
\mathrm{Elm}^+_{D_0,q}\bigl(\Omega_X^1(\log B)\bigr)
\cong \Omega_X^1\bigl(\log(B+D_0)\bigr), 
\end{equation}
i.e., the positive transformation permits a logarithmic pole on differential
forms.
\end{corollary}

\begin{proof}
Set $u_0^{\perp}:=\{ m\in E^\vee \mid m(u_0)=0\}$. Because $D_0$ is not a component of $B$, the filtration of $\Omega_X^1(\log B)$ at $\rho_0$ is dual to the ordinary tangent filtration and is therefore
\begin{equation}
\label{pos-log-cot-filt}
(E^{\vee})^{\rho_0}(a)=
\begin{cases}
E^\vee ,&a\leq-1,\\
u_0^{\perp},&a=0,\\
0,&a\geq1.
\end{cases}
\end{equation}
The two nonzero graded pieces are
\begin{equation}
\gr_{-1}^{\rho_0}(E^\vee)=E^{\vee}/u_0^{\perp},\qquad \gr_0^{\rho_0}(E^\vee)=u_0^{\perp}.
\end{equation}
Under the quotient in \eqref{restr-log-cot-pos-quot}, the first graded piece is the conormal kernel and the second maps isomorphically to the generic fiber of $\Omega_{D_0}^1(\log B_0)$. Thus we have
\begin{equation}
q_{-1}=0, \qquad K_{-1}=E^\vee/u_0^{\perp},\qquad q_0:u_0^{\perp}\xrightarrow{~~\cong~~ } \overline {E^\vee},\qquad K_0=0,
\end{equation}
where $\overline{E^\vee}=(E/\CC u_0)^{\vee}\cong u_0^{\perp}$. The positive-transformation now raises the conormal graded line from weight $-1$ to weight $0$. Explicitly,
\begin{equation}
\bigl((E^\vee)^+\bigr)^{\rho_0}(a)=
\begin{cases}
E^\vee,&a\leq0,\\
0,&a\geq1.
\end{cases}
\end{equation}
For example, at the new index $0$ one has
\begin{equation}
\bigl((E^\vee)^+\bigr)^{\rho_0}(0)=\pi_{-1}^{-1}(K_{-1})=E^\vee.
\end{equation}
All other ray filtrations remain unchanged, and these are precisely the Klyachko filtrations of $\Omega_X^1(\log(B+D_0))$, because adding $D_0$ to the logarithmic boundary changes the $\rho_0$-filtration from \eqref{pos-log-cot-filt} to the single weight-zero filtration above.     
\end{proof}


\section{Further examples of elementary transforms}

\begin{example}\label{general-split}
Let $X=X_\Sigma$ be a smooth toric variety and fix an invariant prime divisor $D=D_{\rho_0}$. Set
\[
\Ee\cong \bigoplus_{j=1}^{m}\Ll_j
\]
be a direct sum of line bundles and choose a subset $J\subseteq\{1,\ldots,m\}$ with $|J|=r$. For the projection onto the selected summands:
\begin{equation}\label{split-project}
q:~\Ee_{|D}\twoheadrightarrow \bigoplus_{j\in J}\Ll_{j|D},
\end{equation}
we have
\begin{equation}\label{general-split-transform}
\mathrm{Elm}_{D,q}(\Ee) \cong \left(\bigoplus_{j\notin J}\Ll_j\right)
\oplus \left(\bigoplus_{j\in J}\Ll_j(-D)\right).
\end{equation}
Indeed, the quotient map $q$ is the direct sum of the zero map on $\Ll_j$ for $j\notin J$ and the restriction map $\Ll_j\rightarrow  i_*(\Ll_{j|D})$ for $j\in J$. For every selected summand, the kernel of the corresponding quotient map is isomorphic to $\Ll_j(-D)$ so that we obtain \eqref{general-split-transform}. The same conclusion can also follow from the filtrations. Write
\[
\Ll_j=\Oo_X\left(\sum_{\rho\in\Sigma(1)}a_{j,\rho}D_\rho\right).
\]
On the generic line $\CC e_j$ in $E$, corresponding to $\Ll_j$, its filtration is
\[
L_j^\rho(i)=
\begin{cases}
\CC e_j,&i\leq a_{j,\rho};\\
0,&i>a_{j,\rho}.
\end{cases}
\]
At the distinguished ray $\rho_0$, the $j$-th summand jumps at $a_{j,\rho_0}$. If $j\in J$, the quotient is nonzero on the corresponding graded line, so Proposition~\ref{elem-transf} lowers the jump by one:
\[
a_{j,\rho_0}\longmapsto a_{j,\rho_0}-1.
\]
This is precisely the filtration change $\Ll_j\mapsto\Ll_j(-D)$. If $j\notin J$, the quotient vanishes on that line, so its filtration is unchanged. All filtrations at rays $\rho\neq\rho_0$ remain unchanged.
\end{example}

\begin{example}\label{tp3}
Let $X=\PP^3$ with ray generators 
\[
u_0=(1,0,0),\qquad u_1=(0,1,0),\qquad u_2=(0,0,1),\qquad u_3=(-1,-1,-1).
\]
Fix $D=D^{\rho_0}=\{x_0=0\}\cong \PP^2$. From the isomorphism
\[
\Tt_{\PP^3|D}\cong \Oo_D(1)\oplus \Tt_D,
\]
one can set $q:\Tt_{\PP^3|D}\rightarrow  \Tt_D$ the natural projection. In the filtration of $\Tt_{\PP^3}$ for $\rho_0$, the nonzero associated graded pieces are
\begin{equation}
\gr_0^{\rho_0}(\Tt_{\PP^3})=E/\CC u_0,
\qquad
\gr_1^{\rho_0}(\Tt_{\PP^3})=\CC u_0.
\end{equation}
The quotient lattice of $D$ is $\overline N=N/\ZZ u_0$ and $\overline E=E/\CC u_0$ is the generic fiber of $\Tt_D$. If $\bar\rho_i$ is generated by the image $\bar u_i$ of $u_i$ for $i=1,2,3$, then
\begin{equation}
(\Tt_D)^{\bar\rho_i}(j)=
\begin{cases}
\overline E ,&j\leq0;\\
\CC \overline{u}_i,&j=1;\\
0,&j\geq2.
\end{cases}
\end{equation}
The generic maps defining the quotient are
\begin{equation}
q_0:E/\CC u_0\xrightarrow{~~\cong~~}\overline{E},
\qquad
q_1:\CC u_0 \longrightarrow 0
\end{equation}
so that $\ker q_0=0$ and $\ker q_1=\CC u_0$. Define
\begin{equation}
\Ee^-:=\mathrm{Elm}_{D, \Tt_D}(\Tt_X)\cong \ker\left(\Tt_{\PP^3}\longrightarrow i_*\Tt_D\right).
\end{equation}
At every ray $\rho\neq\rho_0$, the filtration of $\Ee^-$ agrees with that of $\Tt_{\PP^3}$, while at $\rho_0$ the elementary-transformation formula yields
\begin{equation}
(E^-)^{\rho_0}(i)=\begin{cases}
E,&i\leq-1;\\
\CC u_0 ,&i=0,1;\\
0,&i\geq2.
\end{cases}
\end{equation}
Indeed, we have
\[
(E^-)^{\rho_0}(0)=\pi_0^{-1}(\ker q_0)
=\pi_0^{-1}(0)=\CC u_0=\pi_1^{-1}(\ker q_1)=(E^-)^{\rho_0}(1).
\]
Since the quotient has rank two and $\Oo_{\PP^3}(D)\cong \Oo_{\PP^3}(1)$, we have $c_1(\Ee^-)=2$. This example should be contrasted with the logarithmic transform: quotienting by the normal bundle gives $\Tt_{\PP^3}(-\log D)$, whereas quotienting by the tangent summand $\Tt_D$ gives the subbundle $\mathrm{Elm}_{D, \Tt_d}(\Tt_X)$ above. Finally, note that this elementary transformation bundle is unstable, included in $\Tt_{\PP^3}$ as a subsheaf, due to the existence of a destabilizing subsheaf $\Oo_{\PP^3}(1)$, produced by $\CC u_0$. 
\end{example}

\begin{example}
Use the notation of Example \ref{tp3} and take $\Ee=\bigwedge\nolimits^2 \Tt_{\PP^3}\cong \Omega_{\PP^3}^1(4)$. This is an indecomposable bundle of rank three. From the equivariant splitting $\Tt_{\PP^3|D}\cong  \Tt_D\oplus\Oo_D(1)$, we obtain
\[
\bigwedge\nolimits^2\Tt_{\PP^3|D}\cong \bigwedge\nolimits^2\Tt_D\oplus\bigl(\Tt_D\otimes\Oo_D(1)\bigr)\cong \Oo_D(3)\oplus \Tt_D(1).
\]
Projection onto the second summand gives the equivariant quotient map
\begin{equation}
q: ~\bigwedge\nolimits^2\Tt_{\PP^3|D} \twoheadrightarrow \Tt_D(1),
\end{equation}
where $\Tt_D(1)$ is an indecomposable bundle of rank two on $D$. The filtration on an exterior power is induced from that of the original bundle by
\[
\left(\bigwedge\nolimits^2 E\right)^\rho(i)=\sum_{a+b=i}E^\rho(a)\wedge E^\rho(b).
\]
At $\rho_0$, fix a subspace $U$ of dimension two such that $E=\CC u_0\oplus U$. Since the weights of $\Tt_{\PP^3}$ at this ray are $1$ on $\CC u_0$ and $0$ on $U$, we have
\[
\left(\bigwedge\nolimits^2\Tt_{\PP^3}\right)^{\rho_0}(i)
=
\begin{cases}
\bigwedge^2E ,&i\leq0,\\
u_0\wedge U,&i=1,\\
0,&i\geq2.
\end{cases}
\]
The associated graded pieces are
\[
\gr_0^{\rho_0}(\Ee)=\bigwedge\nolimits^2U\quad  \text{and} \quad \gr_1^{\rho_0}(\Ee)=u_0\wedge U,
\]
where the first has dimension one and corresponds to the summand
$\bigwedge^2\Tt_D\cong \Oo_D(3)$, while the second has dimension two and corresponds to $\Tt_D(1)$. Therefore, the quotient maps satisfy
\[
q_0:\bigwedge\nolimits^2U\longrightarrow0,
\qquad
q_1:v_0\wedge U\xrightarrow{~~\cong~~}\left(\Tt_D(1)\right)_{\eta_D},
\]
where $\eta_D$ is the generic point of $D$. 
Let
\begin{equation}
\Gg:=\mathrm{Elm}_{D, \Tt_D(1)} \left(\bigwedge\nolimits^2\Tt_{\PP^3}\right)\cong \ker\left(
\bigwedge\nolimits^2\Tt_{\PP^3}\longrightarrow i_*\Tt_D(1)\right).
\end{equation}
Since $q_0=0$ and $q_1$ is an isomorphism, the transformed filtration at $\rho_0$ is
\[
G^{\rho_0}(i)= \begin{cases}
\bigwedge^2E ,&i\leq0,\\
0,&i\geq1.
\end{cases}
\]
All other ray filtrations are unchanged. Thus the rank-two step $v_0\wedge U$ at index $1$ is deleted. Indeed, we have
\[
\Gg:=\mathrm{Elm}_{D, \Tt_D(1)} \left(\bigwedge\nolimits^2\Tt_{\PP^3}\right)\cong \Oo_{\PP^3}(2)^{\oplus 3}. 
\]
Set $L_{i\!j}:=\CC (u_i \wedge u_j)$ for $1\le i<j\le 3$ to have
\[
\bigwedge\nolimits^2 E =L_{12}\oplus L_{13}\oplus L_{23}.
\]
For the three rays $\rho\ne \rho_0$, the positive filtration steps are
\[
u_1\wedge E=L_{12}\oplus L_{13},\quad 
u_2\wedge E=L_{12}\oplus L_{23},\quad
u_3\wedge E=L_{13}\oplus L_{23}.
\]
At the ray $\rho_0$, the transformed filtration has no positive step. Thus all four ray filtrations are simultaneously split by $L_{12},L_{13},L_{23}$. Klyachko's equivalence consequently gives the equivariant splitting
\[
\Gg\cong \Oo_{\PP^3}\left(D_{\rho_1}+D_{\rho_2}\right)\oplus\Oo_{\PP^3}\left(D_{\rho_1}+D_{\rho_3}\right)\oplus\Oo_{\PP^3}\left(D_{\rho_2}+D_{\rho_3}\right).
\]
For example, $L_{12}$ has jump $1$ at $\rho_1$ and $\rho_2$, and jump $0$ at $\rho_0$ and $\rho_3$, so it gives
$\Oo_{\PP^3}\left(D_{\rho_1}+D_{\rho_2}\right)$; the other two lines are analogous. 
\end{example}


\end{document}